\documentclass[11pt,oneside]{amsart}
\usepackage{amsxtra}
\usepackage{color}
\usepackage{amsopn}
\usepackage{amsmath,amsthm,amssymb}

\theoremstyle{remark}
\newtheorem{remark}{Remark}[section]

\theoremstyle{definition}
\newtheorem{defin}[remark]{Definition}
\newtheorem{ex}[remark]{Example}

\theoremstyle{theorem}
\newtheorem{teo}[remark]{Theorem}
\newtheorem{prop}[remark]{Proposition}

\newtheorem{lemma}[remark]{Lemma}

\newcommand{\beq}{\begin{equation}}
\newcommand{\eeq}{\end{equation}}
\newcommand{\bqn}{\begin{eqnarray}}
\newcommand{\eqn}{\end{eqnarray}}
\newcommand{\bqne}{\begin{eqnarray*}}
\newcommand{\eqne}{\end{eqnarray*}}

\newcommand{\R}{{\mathbb R}}

\newcommand{\C}{{\mathbb C}}
\newcommand{\Z}{\mathbb Z}

\newcommand{\W}{\wedge}

\newcommand{\f}{\varphi}

\newcommand{\psip}{\psi_+}
\newcommand{\psim}{\psi_-}
\newcommand{\ff}{{\rm f}}
\newcommand{\SU}{{\rm SU}}

\begin{document}
\renewcommand\arraystretch{1.2}
\title[Locally conformal calibrated $G_2$-manifolds]{Locally conformal calibrated $G_2$-manifolds}

\author{Marisa Fern\'andez, Anna Fino and  Alberto Raffero}
 \address{(Marisa Fern\'andez) Universidad del Pa\'i's Vasco, Facultad de Ciencia y Tecnolog\'ia, 
Departamento de Matem\'aticas, Apartado 644,  48080 Bilbao, Spain}
\email{marisa.fernandez@ehu.es}

\address{(Anna Fino, Alberto Raffero) Dipartimento di Matematica ``G. Peano'' \\ Universit\`a di Torino\\
Via Carlo Alberto 10\\
10123 Torino\\ Italy}
\email{annamaria.fino@unito.it, alberto.raffero@unito.it}


\begin{abstract} 
We study conditions for which the mapping torus of a 6-manifold endowed with an $\SU(3)$-structure is a locally conformal calibrated $G_2$-manifold, 
that is, a 7-manifold endowed with a $G_2$-structure $\varphi$ such that $d \varphi = -  \theta \wedge \varphi$ for a closed non-vanishing  1-form $\theta$. 
Moreover, we show that if $(M, \varphi)$ is a compact locally conformal calibrated $G_2$-manifold with $\mathcal{L}_{\theta^{\#}} \varphi =0$, 
where ${\theta^{\#}}$ is the dual of $\theta$ with respect to the Riemannian metric $g_{\varphi}$ induced by $\varphi$,  then $M$ is a fiber bundle over $S^1$ with a coupled 
$\SU(3)$-manifold as fiber.
\end{abstract}

\maketitle
\section{Introduction}
A $G_2$-structure on a  $7$-manifold $M$ can be characterized by the existence of a globally defined 3-form $\varphi$, called  the fundamental 3-form, 
which can be written at each point as
\begin{equation}\label{def:g2-form}
\varphi =e^{127} +e^{347} +e^{567} +e^{135} - e^{146} -e^{236}-e^{245},
\end{equation}
with respect to some local coframe  $\left\{e^1,  \ldots, e^7\right\}$  on $M$. A $G_2$-structure $\varphi$  induces a Riemannian metric
$g_{\varphi}$ and a volume form $dV_{g_\f}$ on $M$ given by
\begin{equation}\label{gfi}
g_{\varphi} (X, Y) dV_{g_\f}=  \frac16i_X \varphi \wedge i_Y \varphi \wedge \varphi,
\end{equation}
for  any pair of vector fields  $X,Y$ on $M$.

The classes of $G_2$-structures can be described in the terms of the exterior derivatives of the fundamental 3-form $\varphi$ and the 4-form 
$*\varphi$ \cite{Bryant2,FG}, 
where $*$ is the Hodge operator defined from $g_\varphi$ and $dV_{g_\f}$. 
In this paper, we focus our attention on the class of {\em locally conformal calibrated} $G_2$-structures, which are characterized by the condition
$$
d \varphi = - \theta \wedge \varphi,
$$
for a closed non-vanishing 1-form $\theta$ also known as the {\em Lee form} of the $G_2$-structure. We refer to a manifold endowed with such a structure as a 
{\em locally conformal calibrated $G_2$-manifold}.

A differential complex for locally conformal calibrated $G_2$-manifolds was introduced in \cite{FU}, where such manifolds were characterized as the ones endowed with a 
$G_2$-structure $\f$ for which the space of differential forms annihilated by $\f$ becomes a differential subcomplex of the de Rham's complex.

Locally conformal calibrated $G_2$-structures $\varphi$ whose underlying Riemannian metric  $g_{\varphi}$ is Einstein were studied in \cite{FR}, where it was shown that in the 
compact case the scalar curvature of $g_\f$ can not be positive. Moreover, in contrast to a result obtained in the compact homogeneous case, a non-compact example of homogeneous 
manifold endowed with a locally conformal calibrated $G_2$-structure whose associated Riemannian metric is Einstein and non Ricci-flat was given. The homogeneous Einstein metric 
is a rank-one extension of a Ricci soliton induced on the 3-dimensional complex Heisenberg group by a left-invariant 
{\em coupled} $\SU(3)$-structure $(\omega, \psi_+)$, i.e., such that  $d \omega =c \psi_+,$ with  $c\in\R-\{0\}$  
(see Definition \ref{def:coupled} for details on coupled $\SU(3)$-structures).

In the general case, it is not difficult to show that the product of a coupled $\SU(3)$-manifold by $S^1$ admits a natural locally conformal calibrated $G_2$-structure. 

In \cite{IPP}, a characterization of compact locally conformal parallel $G_2$-manifolds as fiber bundles over $S^1$ with compact nearly K\"ahler  fiber  
was obtained (see also \cite{Verbitsky}). 

Banyaga in \cite{Banyaga} showed that special types of exact locally conformal symplectic manifolds are fibered over $S^1$ with each fiber carrying a contact form. In this context, 
exact means that the locally conformal symplectic structure $\Omega$ is $d_{\theta}$-exact, i.e., $\Omega = d \alpha + \theta \wedge \alpha$, for a $1$-form $\alpha$ and a 
closed 1-form $\theta$.

Exact locally conformal symplectic structures are precisely the structures called {\it of the first kind} in Vaisman's paper \cite{Vaisman}, where he showed that a manifold endowed with 
such a structure is a 2-contact manifold and has a vertical 2-dimensional foliation. Moreover, if this foliation is regular, then the manifold is a $T^2$-principal bundle over a symplectic 
manifold.

More in general, by \cite[Proposition 3.3]{BK}, every  compact  manifold admitting a generalized contact pair of type $(k,0)$
fibers over the circle with fiber a contact manifold and the monodromy acting by a contactomorphism. 
Note that a contact pair $(\alpha, \beta)$ of type $(k,0)$ induces a locally conformal symplectic form defined by $d \alpha + \alpha \wedge \beta$.
Conversely, every mapping torus of a contactomorphism admits a generalized contact pair of type $(k,0)$ and an induced  locally conformal symplectic  form.  

A theorem by Tischler \cite{Tischler} asserts that the existence of a non-vanishing closed 1-form on a compact manifold $M$ is  equivalent to the condition that $M$ is a mapping torus.  
In the last years, mapping tori of certain diffeomorphisms have been very useful to study some geometric structures. For example, Li in \cite{Li}
proved that any compact cosymplectic manifold is the mapping torus of a symplectomorphism. Furthermore, formality of mapping tori was studied in \cite{BFM}.

A natural problem is then to find a characterization of $d_\theta$-exact locally conformal calibrated $G_2$-manifolds with Lee form $\theta$ and to see under which conditions a locally 
conformal calibrated $G_2$-manifold is  a mapping torus of a special type of $\SU(3)$-manifold or, more in general, it is  fibered over $S^1$ with fiber endowed with a special type of 
$\SU(3)$-structure. 

In Section \ref{SU3tomapcdp}, we will show that if $N$ is a six-dimensional, compact, connected coupled $\SU(3)$-manifold, and  $\nu\colon N \rightarrow N$ 
is a diffeomorphism preserving the $\SU(3)$-structure of $N$, then the mapping torus $N_\nu$ of $\nu$ admits a locally conformal calibrated $G_2$-structure. 
Moreover, in the same section, 
we will show that a result of this kind also holds for a compact nearly K\"ahler $\SU(3)$-manifold, which is a particular case of coupled $\SU(3)$-manifold. 
In detail, we will prove that the mapping torus of a nearly K\"ahler $\SU(3)$-manifold with respect to a diffeomorphism preserving the nearly K\"ahler structure is endowed 
with a locally conformal parallel $G_2$-structure. 

In a similar way as in the paper \cite{Banyaga} on locally conformal symplectic manifolds, in Section \ref{exactlcl} we find some characterizations for locally conformal calibrated 
$G_2$-structures  $\f$ which are  $d_\theta$-exact, that is, such that $\f = d_\theta\omega = d\omega + \theta\W\omega$,  where $\theta$ is the Lee form of $\f$ and $\omega$ is a 
2-form on $M$. 
In fact, for a locally conformal calibrated $G_2$-manifold $(M, \varphi)$, we prove in Proposition \ref{thetadiesis} that if $X$ is the $g_\f$-dual 
vector field of $\theta$ and $\omega$ is the 2-form given by $\omega=  i_{X}\varphi$, then $X$ is an infinitesimal automorphism of $\varphi$ 
$(\mathcal{L}_{X}\varphi = 0)$ if and only if $\theta(X)\varphi = d_\theta\omega$.

Section \ref{lcG2-Liealgebras} is devoted to the construction of Lie algebras admitting a locally conformal calibrated $G_2$-structure 
from six-dimensional Lie algebras endowed with a coupled $\SU(3)$-structure. As a consequence, new examples of compact manifolds with a locally conformal calibrated 
$G_2$-structure are constructed.

Finally, in the last section we obtain a characterization for compact locally conformal calibrated $G_2$-manifolds $(M, \varphi)$
under the assumption $\mathcal{L}_{{\theta}^\#}  \varphi =0$, where ${\theta}^{\#}$ is the dual of  the Lee form $\theta$ with respect to the Riemannian metric $g_{\varphi}$ induced by $
\varphi$. More precisely, in Theorem \ref{teo?}, we show that $M$ is a fiber bundle over 
$S^1$ such that each fiber is equipped with a coupled $\SU(3)$-structure.

\section{Preliminaries}
Given a seven-dimensional manifold $M$ endowed with a $G_2$-structure $\f$,  let $g_\f$ and $dV_{g_\f}$ denote respectively the Riemannian metric and the volume form on $M$ 
induced by $\f$ via the relation \eqref{gfi}.

A manifold endowed with a $G_2$-structure $\f$ is said to be {\em locally conformal calibrated} $G_2$-manifold if 
$$
d\f = - \theta\W\f,
$$
where $\theta$ is a non-vanishing $1$-form which has to be closed. $\theta$ is called the {\em Lee form} associated to the locally conformal calibrated $G_2$-structure $\f$ 
and can  be defined as
$$
\theta = \frac{1}{4} * (* d \varphi \wedge \varphi),
$$
where $*$ is the Hodge star operator defined from the metric $g_{\varphi}$ and the volume form $dV_{g_\f}$.   

If the Lee form vanishes, then the 3-form $\f$ is closed and the $G_2$-structure is called {\em calibrated}.
 
A $G_2$-structure $\f$ is said to be {\em locally conformal parallel} if 
$$
d \varphi = - \theta \wedge \varphi,  \quad d * \varphi = -  \frac{4}{3} \theta \wedge * \varphi.
$$
By \cite[Theorem 3.1]{FI},  given a compact manifold $M$ admitting a $G_2$-structure $\varphi$, there exists a unique 
(up to homothety) conformal  $G_2$-structure $e^{3f} \varphi$ such that the corresponding Lee form is coclosed.  A $G_2$-structure with co-closed Lee form is  
also called a {\em Gauduchon} $G_2$-structure. 

Given a closed 1-form $\tau$ on $M$, we say that a $G_2$-structure $\f$ is $d_\tau$-{\em exact} if there exists a 2-form $\alpha\in\Lambda^2(M)$ such that 
$\f = d\alpha+\tau\W\alpha$. In this case, a simple computation shows that the $G_2$-structure is locally conformal calibrated with Lee form $\tau$. 
The converse is not true in general.

\begin{remark} 
Given a locally conformal calibrated $G_2$-structure $\f$, we can consider the class 
$$
[\f] = \{f\f : f\in C^\infty(M) {\rm\ and\ } f>0\}
$$ 
of locally conformal calibrated $G_2$-structures which are conformally equivalent to $\f$. 
If $d\f = -\theta\W\f$, then $d(f\f) = (d(\ln f) -\theta)\W f\f$ and $\f$ is $d_\theta$-exact if and only if $f\f$ is $d_{d(\ln f) -\theta}$-exact. 
Thus, being $d_\theta$-exact is a conformal property for locally conformal calibrated $G_2$-structures.
\end{remark}

Recall that an $\SU(3)$-structure $(\omega,\psip)$ on a 6-manifold $N$ is said to be {\em half-flat} if both $\omega^2 = \omega\W\omega$ and $\psip$ are closed. 
By \cite{Raffero}, if a half-flat $\SU(3)$-structure on a six-dimensional connected manifold $N$ is such that $d\omega =f\psip$ for some non-vanishing function 
$f \in {\mathcal C}^{\infty} (N)$,  then $f$  has to be  a constant function. This motivates the following

\begin{defin} \label{def:coupled}
Let $N$ be a six-dimensional connected manifold and let $(\omega,\psip)$ be an $\SU(3)$-structure on it. We say that $(\omega,\psip)$ is a {\em coupled} $\SU(3)$-structure 
if $d\omega = c \psip$ for some nonzero real constant $c$ called the {\em coupled constant}.
\end{defin}

Given an $\SU(3)$-structure $(\omega,\psip)$, denote by $J$ the almost complex structure induced by $\psip$, 
by $\psim = J\psip$ and by  $h(\cdot,\cdot)  =  \omega(\cdot,J\cdot)$ the induced  Riemannian metric. 
By  \cite{Schulte}, since an $\SU(3)$-structure on a six-dimensional manifold $N$ is characterized by a pair of stable, compatible ad normalized forms 
$(\omega,\psip)\in\Lambda^2(N)\times\Lambda^3(N)$ inducing a Riemannian metric and since the construction of $J, \psim$ and $h$ is invariant, a 
diffeomorphism preserving the forms $\omega$ and $\psip$ is actually an automorphism of the $\SU(3)$-structure itself and, in particular, an isometry. 
As a special case of this, if we have a coupled $\SU(3)$-structure with $d\omega = c\psip$ for some nonzero constant $c$, then every 
diffeomorphism of $N$ preserving $\omega$ is an automorphism of the $\SU(3)$-structure and an isometry.

\begin{remark} \label{rmcoupledconstant}
Note that if $(\omega,\psip)$ is a {\em nearly K\"ahler} $\SU(3)$-structure $(d\omega = 3\psip, d\psim = -2\omega^2)$, then it is in particular a coupled $\SU(3)$-structure. 
Moreover, if $(\omega,\psip)$ is a coupled $\SU(3)$-structure with $d\omega = c \psip$, then the pair $\hat\omega := k^2\omega, \hat{\psi}_+:=k^3\psip$, 
where $k$ is a nonzero real constant, is still a coupled $\SU(3)$-structure with $d\hat\omega = \frac{c}{k}\hat{\psi}_+$. 
As a consequence, it is always possible to find a coupled structure having coupled constant $c=1$.
Changing the constant $c$ of a coupled $\SU(3)$-structure in the way just described, the almost complex structure $J$ is preserved, while the
Riemannian metric $h$ is rescaled by $k^2$.
\end{remark}

The following proposition shows a first link between coupled $\SU(3)$-structures and locally conformal calibrated $G_2$-structures.
\begin{prop}
Let $N$ be a six-dimensional connected  manifold endowed with a coupled $\SU(3)$-structure $(\omega,\psip)$ with coupled constant $c$, 
then the cylinder $(N \times \R, h+dr^2)$ over $N$ has a locally conformal calibrated $G_2$-structure given by $\f = \omega\W dr + \psip$ 
and such that $g_\f = h+dr^2$. 
Moreover, if $c\neq3$, also the cone over $N$, $(N\times (0,+\infty), r^2h + dr^2)$, 
is endowed with a locally conformal calibrated $G_2$-structure $\f = r^2\omega\W dr + r^3\psip$ whose associated metric $g_\f$ is the cone's one. 
\end{prop}

\section{Mapping torus}\label{SU3tomapcdp}
Let $N$  be a  compact  manifold and $\nu: N \rightarrow N$ a diffeomorphism. 
The {\em mapping torus} $N_{\nu}$ of $\nu$ is the quotient space of $N \times [0,1]$ in which any point $(p, 0)$ is identified with $(\nu(p), 1)$. 
$N_{\nu}$ is naturally a smooth manifold, since it is the quotient
of $N\times\mathbb{R}$ by the infinite cyclic group generated by the diffeomorphism $(p,t) \mapsto (\nu(p),t+1)$. The natural
map $\pi \colon N_{\nu} \to S^1$ defined by $\pi(p,t)=e^{2\pi it}$ is the projection
of a locally trivial fiber bundle (here we think $S^1$ as the interval $[0,1]$ with identified end points). 
Thus, any  $\nu$-invariant form $\alpha$ on $N$ defines a form $\widetilde{\alpha}$ on $N_{\nu}$
since the pullback of $\alpha$ to $N\times\mathbb{R}$ is invariant by the diffeomorphism $(p,t) \mapsto  (\nu(p),t+1)$. 
For the same reason, the 1-form $dt$ on $\R$, where $t$ is the coordinate on $\R$,  induces a closed 1-form $\eta$ on $N_{\nu}$.
Moreover, on $N_\nu$ we have a distinguished vector field $\xi$ induced by the vector field $\frac{d}{dt}$. This vector field is such that $\eta(\xi) = 1$.

In the case of a compact manifold endowed with a coupled $\SU(3)$-structure $(\omega,\psip)$ we can prove the following 

\begin{prop} \label{propmapptorusconf}
Let $N$ be a six-dimensional, compact,  connected manifold endowed with a coupled $\SU(3)$-structure $(\omega,\psip)$ with $d\omega = c\psip$ 
and let $\nu : N \rightarrow N$ be a diffeomorphism such that $\nu^*\omega = \omega$. Then the mapping torus $N_\nu$ admits a locally conformal calibrated 
$G_2$-structure $\f$.  Moreover,  ${\mathcal{L}}_\xi \f = 0.$

\end{prop}
\proof
We have the following situation
$$\renewcommand\arraystretch{1.5}
\begin{array}{ccccc}
N & \stackrel{p_1}{\longleftarrow} & N\times [0,1] & \stackrel{q}{\longrightarrow} & N_\nu \\
    &                           &\downarrow p_2  &                             & \downarrow \pi \\
    &                           & [0,1]                & \stackrel{\Pi}{\longrightarrow} & S^1 
\end{array}
$$
where $p_1$ and $p_2$ are the projections from $N\times [0,1]$ on the first and on the second factor respectively, $q$ is the quotient map, 
$\pi$ is the fibration map and $\Pi(t) = e^{2\pi i t}$.

Let us observe that $p_1^*(\omega) \in \Lambda^2(N\times [0,1])$ and $p_1^*(\psi_\pm) \in \Lambda^3(N\times [0,1])$. 
Then, since $\nu^*\omega = \omega$ and (as a consequence) $\nu^*\psi_\pm = \psi_\pm$, 
we can glue up these pullbacks and obtain a 2-form $\tilde\omega\in\Lambda^2(N_\nu)$ and 3-forms $\tilde\psi_\pm\in\Lambda^3(N_\nu)$ 
satisfying the same relations that hold on $N$.
In particular, since $d (p_1^* \omega) = c p_1^* \psi_{+}$, we have 
$$
d\tilde{\omega}= c\tilde{\psi}_+,  \quad d \tilde \omega^2 =0.
$$
Using the closed 1-form $\eta$ defined on $N_\nu$ we have that 
$$
\f = \eta \W \tilde{\omega} + \tilde{\psi}_+
$$
is a 3-form on $N_\nu$ defining a $G_2$-structure on it. Moreover,
$$
d\f =  d\eta\W\tilde{\omega} - \eta\W d\tilde{\omega} +d\tilde{\psi}_+ = -c\eta\W\tilde{\psi}_+ = -c\eta\W\f,
$$ 
that is, $\f$ is locally conformal calibrated with Lee form $\theta=c\eta$.

Since both $\tilde\omega$ and $\tilde\psi_+$ derive from differential forms defined on $N,$ we have $i_\xi \tilde\omega = 0$ and $i_\xi \tilde\psi_+ = 0$. 
From these conditions it follows that $i_\xi \f = \tilde\omega$ and, consequently,  $\mathcal{L}_\xi \f = 0.$ 
\endproof

\begin{remark}
As we mentioned before, if $(\omega, \psi_{+})$ is a pair of 
compatible, normalized, stable forms on $N,$ then the remaining tensors 
appearing in the definition of an $\SU(3)$-structure, namely the almost complex structure $J$, the 3-form $\psi_{-}=J\psi_{+}$ and the Riemannian metric $h$, 
can be completely determined from $(\omega, \psip)$. 
Thus, the diffeomorphism $\nu$ preserves not only the 2-form $\omega$ and the 3-form $\psi_{+}$ on $N$ but also
the Riemannian metric $h$ of the $\SU(3)$-structure $(\omega, \psi_{+})$.
Hence, (globally) the metric $g_{\varphi}$ on $N_\nu$ is
$g_{\varphi}=\tilde{h} + \eta^2$ and we have that $\xi$ is the vector field dual to $\eta$ with respect to $g_{\varphi}$.
\end{remark}

\begin{ex}  The previous proposition can be applied to compact nilmanifolds admitting a coupled $\SU(3)$-structure. Nilpotent Lie algebras admitting a coupled SU(3)-structure  
were classified in \cite{FR}. One of these is the Iwasawa Lie algebra, i.e., the Lie algebra of the complex Heisenberg group of complex dimension 3
$$
G = \left \{  \left ( \begin{array}{ccc} 1&z_1&z_3\\[3pt] 0&1&z_2\\ 0&0&1 \end{array} \right),   \, z_i \in \C, \,  i = 1,2,3  \right \}.
$$ 
This complex Lie group admits a co-compact discrete subgroup $\Gamma$ 
which is defined as the subgroup of $G$ for which $z_i$  are Gaussian integers. 
Consider the automorphism
$$
\nu:G \rightarrow G, \qquad \left ( \begin{array}{ccc} 1&z_1&z_3\\[3pt] 0&1&z_2\\ 0&0&1 \end{array} \right) \overset{\nu}{\mapsto} \left (  \begin{array}{ccc} 1&z_1& - i z_3\\[3pt] 0&1&- i 
z_2\\ 0&0&1 \end{array} \right).
$$
Then, if we denote by
$$
e^1 + i e^2 = d z_1, \quad e^3 + i e^4 = d z_2,  \quad e^5 + i e^6 = - d z_3 + z_1 \wedge dz_2,
$$
we have
$$
\nu^* (e^1) = e^1, \quad \nu^* (e^2) = e^2,  \quad \nu^* (e^3) = e^4, \quad \nu^* (e^4) = - e^3, \quad   \nu^* (e^5) = e^6, \quad \nu^* (e^6) = - e^5.
$$
The nilmanifold $N = \Gamma\backslash G$ admits  the  coupled $\SU(3)$-structure  $(\omega, \psi_+)$ defined by
$$
\omega = e^{12} + e^{34} - e^{56}, \qquad \psi^+ = e^{136} - e^{145} - e^{235} - e^{246}.
$$
Since $\nu^* \omega = \omega$, we can apply Proposition \ref{propmapptorusconf} and the mapping torus $N_{\nu}$ admits a locally conformal calibrated $G_2$-structure.
\end{ex}

\begin{remark}
Since by Remark \ref{rmcoupledconstant} we can always suppose that the coupled constant is $c=1$, the form $\f = \eta \W \tilde{\omega} + \tilde{\psi}_+$ is  $d_\eta$-exact with $\f=d_\eta\tilde\omega$.
\end{remark}

As we observed in Remark \ref{rmcoupledconstant}, a special case of coupled $\SU(3)$-structures is given by the nearly K\"ahler $\SU(3)$-structures. 
In this case, we can prove what follows.
\begin{prop} 
Let $N$ be a six-dimensional, compact, connected manifold endowed with a nearly K\"ahler $\SU(3)$-structure $(\omega,\psip)$ 
and let $\nu : N \rightarrow N$ be a diffeomorphism such that $\nu^*\omega = \omega$. Then the mapping torus $N_\nu$ admits a locally conformal parallel $G_2$-structure.
\end{prop}
\proof
As in Proposition \ref{propmapptorusconf}, we can define the differential forms  $\tilde\omega\in\Lambda^2(N_\nu)$ and $\tilde\psi_\pm\in\Lambda^3(N_\nu)$, 
which in this case satisfy the relations
$$
\begin{array}{rcl}
d\tilde\omega & = & 3\tilde\psi_+,\\
d\tilde\psi_- & = & -2\tilde\omega^2.
\end{array}
$$
The positive 3-form
$$\f = \eta \W \tilde{\omega} + \tilde{\psi}_+$$
defines a $G_2$-structure on $N_\nu$ with Hodge dual
$$*\f  = \tilde\psi_- \W \eta + \frac12 \tilde\omega^2.$$
It follows from computations that
$$
\begin{array}{rcl}
d\f & = & 3(-\eta)\W\f, \\
d*\f & = & 4(-\eta)\W*\f.
\end{array}
$$
Therefore, $\f$ is a locally conformal parallel $G_2$-structure defined on $N_\nu$.
\endproof  
\begin{ex}
Consider the six-dimensional compact manifold $S^3\times S^3$. As a Lie group it is $\SU(2) \times \SU(2)$ and its Lie algebra is $\frak{su}(2)\oplus\frak{su}(2)$. 
Let $\{e_1, e_2, e_3\}$ denote the standard basis for the first copy of $\frak{su}(2)$, let $\{e_4, e_5, e_6\}$ denote it for the second one and let 
$\left\{e^1,e^2,e^3\right\}$ and $\left\{e^4,e^5,e^6\right\}$ denote their dual bases. The structure equations of $\frak{su}(2)\oplus\frak{su}(2)$ are:
$$\renewcommand\arraystretch{1.4}
\begin{array}{c}
de^1 = e^{23}, \quad de^2 = e^{31}, \quad de^3 = e^{12},\\
de^4 = e^{56},\quad  de^5 = e^{64},\quad  de^6 = e^{45}.
\end{array}
$$
On $\frak{su}(2)\oplus\frak{su}(2)$ we have a pair of stable, compatible, normalized forms 
$$
\begin{array}{rcl}
\omega & = & -\frac{\sqrt{3}}{18}\left(e^{14}+e^{25}+e^{36}\right),\\
\psip &=& \frac{\sqrt{3}}{54}\left(-e^{234}+e^{156}+e^{135}-e^{246}-e^{126}+e^{345}\right),
\end{array}
$$
defining an $\SU(3)$-structure which is nearly K\"ahler and induces a left-invariant nearly K\"ahler $\SU(3)$-structure on $S^3\times S^3$. 

Let $\nu : S^3\times S^3 \rightarrow S^3\times S^3$  be the diffeomorphism such that
$$\nu^*e^1= e^1, \quad \nu^*e^2= e^3, \quad \nu^*e^3= -e^2, \quad \nu^*e^4= e^4, \quad \nu^*e^5= e^6, \quad \nu^*e^6= -e^5.$$
$\nu$ preserves $\omega$, therefore the mapping torus $(S^3 \times S^3)_\nu$ is endowed with a locally conformal parallel $G_2$-structure by the previous proposition.
\end{ex}

A characterization of compact locally conformal parallel $G_2$-manifolds as fiber bundles over $S^1$  with compact nearly K\"ahler  fiber was obtained in \cite{IPP} (see also 
\cite{Verbitsky}). It was also shown there that for compact seven-dimensional  locally conformal parallel $G_2$-manifolds $(M, \varphi)$ with co-closed Lee form $\theta$,  
the Lee flow preserves the Gauduchon $G_2$-structure, i.e., $\mathcal{L}_{\theta^\#} \varphi =0$, where ${\theta^\#}$ is the dual of $\theta$ with respect to $g_{\varphi}$.  
In the next section, we will characterize the locally conformal calibrated  $G_2$-structures such that   $\mathcal{L}_{\theta^\#} \varphi =0$.

\section{Exact locally conformal calibrated  $G_2$-manifolds } \label{exactlcl} 

In a similar way as in the paper \cite{Banyaga} on locally conformal symplectic manifolds, 
we can find some characterizations for $d_\theta$-exact locally conformal calibrated $G_2$-structures $\f$ with Lee form $\theta$.

We recall that $X\in\mathfrak{X}(M)$ is a {\em conformal infinitesimal automorphism} of $\f$ if an only if there exists a smooth function $\rho_X$ on $M$ 
such that $\mathcal{L}_X\f = \rho_X\f$. If $\rho_X\equiv0$, then $X$ is a {\em conformal automorphism} of $\f$.
We start by proving the following
\begin{lemma}\label{BLR1}
Let $(M, \varphi)$ be a locally conformal calibrated $G_2$-manifold with Lee form $\theta$.  A vector field $X\in\mathfrak{X}(M)$ is a conformal infinitesimal automorphism of $\f$ 
if and only if there exists a smooth function $\ff_X\in C^\infty(M)$ such that $d_\theta\omega = \ff_X\f$, where $\omega=i_X\f$. 
Moreover, if $M$ is connected, $\ff_X$ is constant.
\end{lemma}
\proof
Let us compute the expression of the Lie derivative of $\f$ with respect to $X$ 
$$
\begin{array}{lcl}
\mathcal{L}_X\f &=& d(i_X\f) +i_X(d\f)\\
                            &=& d\omega + i_X(-\theta\W\f)\\
                            &=& d\omega - \theta(X)\f + \theta\W(i_X\f)\\
                            &=& d\omega+\theta\W\omega  - \theta(X)\f\\
                            &=& d_\theta\omega  - \theta(X)\f,
\end{array}
$$ 
where $\omega = i_{X} \varphi.$
Therefore, $X$ is a conformal infinitesimal automorphism of $\f$ with $\mathcal{L}_X\f = \rho_X\f$ if and only if 
$d_\theta\omega = \ff_X\f,$
where $\ff_X$ is a smooth real valued function on $M$ such that $\ff_X = \rho_X + \theta(X)$.

Suppose now that $M$ is connected and let $X$ be a conformal infinitesimal automorphism of $\f$, as we have just shown $d_\theta\omega = \ff_X\f$ for some $\ff_X\in C^\infty(M)$.
We have
$$
\begin{array}{lcl}
0 & = & d_\theta(d_\theta\omega)\\
   & = & d_\theta(\ff_X\f)\\
   & = & d(\ff_X\f) + \theta\W(\ff_X\f)\\
   & = & d\ff_X\W\f +\ff_Xd\f + \ff_X(\theta\W\f)\\
   & = & d\ff_X\W\f +\ff_Xd\f -\ff_Xd\f\\
   & = & d\ff_X\W\f.
\end{array}
$$
Since the linear mapping $\W\f : \Lambda^1(M)\rightarrow\Lambda^4(M)$ is injective, we obtain that $d\ff_X = 0$ and from this the assertion follows.
\endproof

\begin{remark}
It is worth emphasizing here that if $X$ is a conformal infinitesimal automorphism of $\f$ with $\ff_X$ a nonzero constant, then $\f$ is $d_\theta$-exact. Indeed 
$$
\f = \frac{1}{\ff_X}d_\theta\omega = d_\theta\left(\frac{\omega}{\ff_X}\right).
$$
\end{remark}

Recall the result contained in \cite{Lin}: 
\begin{lemma}[\cite{Lin}]
Let $M$ be a seven-dimensional, compact manifold. Then for any $G_2$-structure $\f$ on $M$, any vector field $X\in\mathfrak{X}(M)$ and $f\in C^\infty(M)$ it holds
$$\int_M{\mathcal{L}_X\f \W *f\f} = -3 \int_M{df \W *X^\flat}.$$
\end{lemma}

From this Lemma with $f\equiv1$ and $X$ conformal infinitesimal automorphism of $\f$ with $\mathcal{L}_X\f = \rho_X\f$ we have
$$\int_M{\rho_X dV_{g_{\f}}} = 0.$$
Thus, thinking at the proof of Lemma \ref{BLR1}, we get 
$$\int_M{\theta(X) dV_{g_{\f}}} = \int_M{\ff_X dV_{g_{\f}}} = \ff_X \rm{Vol}(M),$$
that is, the Riemannian integral of the function $\theta(X)$ over $M$ is constant.

In conclusion, we can prove the following characterization for a $d_{\theta}$-exact locally conformal calibrated $G_2$-structure.

\begin{prop}\label{thetadiesis}
Let $(M,\f)$ be a connected  locally conformal calibrated $G_2$-manifold with Lee form $\theta$. 
Let $X$ be the $g_\f$-dual vector field of $\theta$, i.e., $\theta(\cdot) = g_\f(X,\cdot)$, and define the 2-form $\omega :=  i_X\f$. Then
$\mathcal{L}_X\f = 0$ if and only if $\theta(X)\f = d_\theta\omega$. 
Moreover, if $\mathcal{L}_X\f = 0$, then $\theta(X) = |X|^2$ is a nonzero constant.
\end{prop}
\proof
We have
$$
\begin{array}{lcl}
\mathcal{L}_X\f &=& d(i_X\f) + i_Xd\f\\
                         &=& d\omega +i_X(-\theta\W\f)\\
                          &=& d\omega -\theta(X)\f + \theta\W\omega.
\end{array}
$$
Therefore, $\mathcal{L}_X\f = 0$ if and only if $\theta(X)\f = d_\theta\omega$. 

If  $\mathcal{L}_X\f = 0$, from Lemma \ref{BLR1} we have that $\theta(X) = |X|^2$ is a nonzero constant,  since $\theta(X)\f = d_\theta\omega$ and  $X = \theta^\#$, where the map 
$\cdot^\# : \Lambda^1(M)\rightarrow \mathfrak{X}(M)$ is an isomorphism. \endproof
 
\section{Locally conformal calibrated $G_2$ Lie algebras } \label{lcG2-Liealgebras} 

In this section, we show that locally conformal calibrated $G_2$-structures defined on seven-dimensional Lie algebras
are closely related to coupled $\SU(3)$-structures on six-dimensional Lie algebras. 
This generalizes the result proved in \cite{Manero-thesis} for calibrated $G_2$-structures on seven-dimensional Lie algebras from symplectic half-flat Lie algebras. 
First, we need to recall some definitions and results about $\SU(3)$- and $G_2$-structures on Lie algebras.

Let $\mathfrak{g}$ be a seven-dimensional Lie algebra. A $G_2$-structure on $\mathfrak{g}$ is a 3-form $\varphi$ on $\mathfrak{g}$
which can be written as in \eqref{def:g2-form} with respect to some basis $\left\{e^1, \dots, e^7\right\}$ 
of the dual space $\mathfrak{g}^*$ of $\mathfrak{g}$. $\varphi$ is said to be a {\em locally conformal calibrated $G_2$-structure} on $\mathfrak{g}$ if 
$$
d\varphi=-\theta\wedge\varphi,
$$
for some closed 1-form $\theta$ on $\mathfrak{g}$, where $d$ denotes the Chevalley-Eilenberg differential on $\mathfrak{g}^*$.

We say that a six-dimensional Lie algebra $\mathfrak{h}$ has an $\SU(3)$-structure if there exists a pair $(\omega, \psi_+)$
of forms on $\mathfrak{h}$, where $\omega$ is a 2-form and $\psi_+$ a 3-form, which can be expressed as
$$
\omega=e^{12} + e^{34} + e^{56}, \qquad \psi^+ = e^{135} - e^{146} - e^{236} - e^{245},
$$
with respect to some basis $\left\{e^1, \dots, e^6\right\}$ of the dual space $\mathfrak{h}^*$. If $\left\{e^1, \dots, e^6\right\}$ is such a basis,
the dual basis $\left\{e_1, \dots, e_6\right\}$ of $\mathfrak{h}$ is called {\em $\SU(3)$-basis}. An $\SU(3)$-structure
$(\omega, \psi_+)$ on $\mathfrak{h}$ is said to be a {\em coupled} $\SU(3)$-structure if 
$$
\hat{d}\omega=c\psi_+
$$
for some nonzero real constant $c$, where $\hat{d}$ is the Chevalley-Eilenberg differential on $\mathfrak{h}^*$.

If $\mathfrak{h}$ is a six-dimensional Lie algebra and $D\in {\rm Der}(\mathfrak{h})$ is a derivation of $\mathfrak{h}$, then the vector space
\begin{equation*}
\mathfrak{g}=\mathfrak{h}\oplus_{D} \mathbb{R}\xi
\end{equation*}
is a Lie algebra with the Lie bracket given by 
\begin{equation}\label{bracket:der}
[U,V]=[U,V]|_{\mathfrak{h}}, \qquad [\xi, U]=D(U),
\end{equation}
for any $U,V \in \mathfrak{h}$.

It is well known that there exists a real representation of the $3\times3$ complex matrices via 
\begin{equation*}
\rho:  \mathfrak{gl}(3, \mathbb{C}) \longrightarrow  \mathfrak{gl}(6, \mathbb{R}).
\end{equation*}
More in detail, if $A\in \mathfrak{gl}(3, \mathbb{C})$, then $\rho(A)$ is the matrix $(B_{ij})_{i,j=1}^3$ with 
\begin{equation*}
B_{ij}=\left(
\begin{array}{cc}
 {\rm Re} A_{ij} & {\rm Im} A_{ij}   \\
 -{\rm Im} A_{ij} & {\rm Re} A_{ij}  \\
\end{array}
\right),
\end{equation*}
where $A_{ij}$ is the $(i,j)$ component of $A$.

Now, suppose that $(\omega, \psi_{+})$ is a coupled $\SU(3)$-structure on a six-dimensional Lie algebra $\mathfrak h$,
and let $D$ be a derivation of $\mathfrak h$ such that $D=\rho(A)$, where $A\in \mathfrak{sl}(3, \mathbb{C})$.  
Then, the matrix representation of $D$ with respect to an $\mathrm{SU}(3)$-basis 
$\left\{e_1,\dots, e_6\right\}$ of $\mathfrak h$ is
\begin{equation}\label{sl(3,C)}
D=
\left(
\begin{array}{cc|cc|cc}
 a_{11} & a_{12} & a_{13} & a_{14} & a_{15} & a_{16}  \\
 -a_{12} & a_{11} & -a_{14} & a_{13} & -a_{16} & a_{15}  \\\hline
 a_{31} & a_{32} & a_{33} & a_{34} & a_{35} & a_{36}  \\
 -a_{32} & a_{31} & -a_{34} & a_{33} & -a_{36} & a_{35}  \\\hline
 a_{51} & a_{52} & a_{53} & a_{54} & -a_{11}-a_{33} & -a_{12}-a_{34}  \\
 -a_{52} & a_{51} & -a_{54} & a_{53} & a_{12}+a_{34} & -a_{11}-a_{33}  \\
\end{array}
\right), 
\end{equation}
where $a_{ij} \in \mathbb{R}$.

\begin{prop} \label{liealgebrasG2}
Let $(\omega,\psi_{+})$ be a coupled $\SU(3)$-structure on a Lie algebra $\mathfrak{h}$ of dimension $6$ and let $D = \rho (A)$, 
$A \in \mathfrak{sl}(3, \mathbb{C}),$ be 
a derivation of $\mathfrak{h}$ whose matrix
representation with respect to an $\mathrm{SU}(3)$-basis 
$\left\{e_1,\dots, e_6\right\}$ of $\frak h$ is as in \eqref{sl(3,C)}. Then, the Lie algebra
\begin{equation*}
\mathfrak{g}=\mathfrak{h}\oplus_D\mathbb{R}\xi,
\end{equation*}
with the Lie bracket given by \eqref{bracket:der}, has a locally conformal calibrated $G_2$-structure.
\end{prop}

\begin{proof}
We define the $G_2$ form $\varphi$ 
on $\mathfrak{g}=\mathfrak{h}\oplus_{D} \mathbb{R}\xi$ by
\begin{equation}\label{G2-form-algebra}
\varphi=\omega\wedge \eta +\psi_+,
\end{equation}
where $\eta$ is the 1-form on $\mathfrak{g}$ such that $\eta(X)=0$ for all $X\in \mathfrak{h}$ and $\eta(\xi)=1$.
We will see that 
$$
d\varphi=-c\eta\wedge\varphi,
$$
where $c$ is the coupled constant of the coupled $\SU(3)$-structure on $\mathfrak{h}$.

Suppose that $X,Y,Z,U\in\mathfrak{h}$. Then, it is clear that $(d\omega\W\eta)(X, Y, Z, U)=0$, and
\begin{equation*}
d\varphi(X, Y, Z, U)=d\psi_+(X, Y, Z, U) = \hat{d}\psi_+(X, Y, Z, U) =0,
\end{equation*}
since $\psi_+$ is $\hat{d}$-closed and $d\eta=0$.

Let us consider $X,Y,Z \in \mathfrak{h}$. Using \eqref{G2-form-algebra}, we have
\begin{equation*}
\begin{array}{rcl}
d\varphi(X,Y,Z,\xi)&=&-\varphi([X,Y],Z,\xi)+\varphi([X,Z],Y,\xi)-\varphi([X,\xi],Y,Z)\\
                            &  &-\varphi([Y,Z],X,\xi)+\varphi([Y,\xi],X,Z)-\varphi([Z,\xi],X,Y)\\
                            &=&-\omega([X,Y],Z)+\omega([X,Z],Y)-\omega([Y,Z],X)\\
                            & &-\psi_+([X,\xi],Y,Z)+\psi_+([Y,\xi],X,Z)-\psi_+([Z,\xi],X,Y)\\
                            &=&d\omega(X,Y,Z)+\psi_+(D(X),Y,Z)+\psi_+(X,D(Y),Z)+\psi_+(X,Y,D(Z)).
\end{array}
\end{equation*}

Taking into account the expressions of $D$ and
$\psi_+$ in terms of the $\mathrm{SU}(3)$-basis $\{e_1,\dots, e_6\}$, it is easy to check that
\begin{equation*}
\psi_+(D(e_i),e_j,e_k)+\psi_+(e_i,D(e_j),e_k)+\psi_+(e_i,e_j,D(e_k))=0,
\end{equation*}
for every triple $(e_i, e_j, e_k)$ of elements of the $\SU(3)$-basis.
Therefore, 
\begin{equation*}
d\varphi(X,Y,Z,\xi)=d\omega(X,Y,Z)= \hat{d}\omega(X,Y,Z)=c\psi_+(X,Y,Z).
\end{equation*}
Using \eqref{G2-form-algebra} again, we get
\begin{equation*}
d\varphi(X,Y,Z,\xi)=-(c\eta\wedge\varphi)(X,Y,Z,\xi),
\end{equation*}
which completes the proof that the 3-form $\varphi$ 
given by \eqref{G2-form-algebra} defines a locally conformal calibrated $G_2$-structure on $\mathfrak{g}$.
\end{proof}

As an application of the previous proposition, we describe two examples of 
non-isomorphic solvable Lie algebras endowed with a locally conformal calibrated $G_2$-structure. 
They are obtained considering two different derivations on the Iwasawa Lie algebra.

\begin{ex}\label{excpt}
Consider the six-dimensional Iwasawa Lie algebra $\mathfrak{n}$ and let $\{e_1, \dots, e_6\}$ denote an $\SU(3)$-basis for it. 
With respect to the dual basis $\left\{e^1,\ldots,e^6\right\}$, the structure equations of $\mathfrak{n}$ are
$$
(0,0,0,0,e^{14}+e^{23},e^{13}-e^{24})
$$
and the pair
$$
\omega=e^{12} + e^{34} + e^{56}, \qquad \psi_+ = e^{135} - e^{146} - e^{236} - e^{245},
$$
defines a coupled $\SU(3)$-structure on $\mathfrak{n}$ with $\hat{d}\omega=-\psip$.

Let $D$ be the derivation of $\mathfrak{n}$ defined as follows
$$
De_1 = -e_3, \quad De_2 = -e_4, \quad De_3 = e_1, \quad De_4 = e_2, \quad De_5 =0, \quad De_6=0.
$$
The Lie algebra $\mathfrak{s}=\mathfrak{n}\oplus_{D} \mathbb{R}e_7$
has the following structure equations with respect to the basis $\left\{e^1,\ldots,e^6,e^7\right\}$ of $\mathfrak{s}^*$
$$
(e^{37},e^{47},-e^{17},-e^{27},e^{14}+e^{23},e^{13}-e^{24},0).
$$
By Proposition \ref{liealgebrasG2}, the 3-form 
$$
\f = \omega\W e^7 + \psip 
$$
defines then a locally conformal calibrated $G_2$-structure on $\mathfrak{s}$ with Lee form $\theta=-e^7$. 

Let $S$ denote the simply connected solvable Lie group with Lie algebra $\mathfrak{s}$, let $N$ denote the simply connected nilpotent Lie group such that 
${\rm Lie}(N) = \mathfrak{n}$ and let $e\in N$ denote the identity element. 
Observe that $S = \R\ltimes_\mu N,$ where $\mu$ is the unique smooth action of $\R$ on $N$ such that 
$\mu(t)_{*e} = {\rm exp}(tD)$, for any $t\in\R$, and where ${\rm exp}$ denotes the map
${\rm exp} : {\rm Der}(\frak{n}) \rightarrow {\rm Aut}(\frak{n})$. Hence, $S$ is {\em almost nilpotent} in the sense
of \cite{GVO}.

Now, in order to show a lattice of $S$ we proceed as follows. The $\SU(3)$-basis $\{e_1,\ldots,e_6\}$ we considered is a rational basis for $\mathfrak{n}$ 
and with respect to this basis we have 
$$
{\rm exp}(t D) = 
\left[ \begin {array}{cccccc} \cos(t)&0&\sin(t)&0&0&0
\\ \noalign{\medskip}0&\cos(t)&0&\sin(t)&0&0
\\ \noalign{\medskip}-\sin(t)&0&\cos(t)&0&0&0
\\ \noalign{\medskip}0&-\sin(t)&0&\cos(t)&0&0
\\ \noalign{\medskip}0&0&0&0&1&0
\\ \noalign{\medskip}0&0&0&0&0&1
\end {array} \right] .
$$
In particular, ${\rm exp}(\pi D)$ is an integer matrix. Therefore,  ${\rm exp}^N(\Z\langle e_1,\ldots,e_6\rangle)$ is a lattice of $N$ preserved by  $\mu (\pi)$ and, consequently,  
$$
\Gamma = \pi \Z \ltimes_\mu {\rm exp}^N(\Z\langle e_1,\ldots,e_6\rangle)
$$
is a lattice in $S$ (see \cite{Bock}).
Thus, the compact quotient $\Gamma \backslash S$ is a compact solvmanifold endowed with an invariant locally conformal calibrated $G_2$-structure $\f$.  
Moreover, if $X=-e_7$ denotes the $g_\f$-dual vector field of the Lee form $\theta=-e^7$, then $i_X\f = -\omega$, $\mathcal{L}_X\f=0$ and $\f = d_\theta(-\omega)$, 
as we expected from Proposition \ref{thetadiesis}.
\end{ex}

\begin{ex}
Let us consider the coupled $\SU(3)$-structure $(\omega,\psip)$ on the Iwasawa Lie algebra $\mathfrak{n}$ described in the previous example and the derivation 
$D\in{\rm Der}(\mathfrak{n})$ given by
$$
De_1 =  2 e_3, \quad De_2 = 2e_4, \quad De_3 = e_1, \quad De_4 = e_2, \quad De_5 =0, \quad De_6=0,
$$  
with respect to the $\SU(3)$-basis $\{e_1, \dots, e_6\}$ of $\mathfrak{n}$.
Then, the Lie algebra $\mathfrak{q}=\mathfrak{n}\oplus_{D} \mathbb{R}e_7$
has the following structure equations with respect to the basis $\left\{e^1,\ldots,e^6,e^7\right\}$ of $\mathfrak{q}^*$
$$
(e^{37},e^{47},2e^{17},2e^{27},e^{14}+e^{23},e^{13}-e^{24},0).
$$
The 3-form 
$$
\f = \omega\W e^7 + \psip 
$$
defines a locally conformal calibrated $G_2$-structure on $\mathfrak{q}$ with Lee form $\theta=-e^7$ by Proposition \ref{liealgebrasG2}.

In this case, it is easy to check that if $X=-e_7$ denotes the $g_\f$-dual vector field of $\theta$, 
then $\mathcal{L}_X\f\neq0$ and, accordingly with Proposition \ref{thetadiesis}, 
$\f \neq d_\theta(i_X\f)$.
However, $\f$ is $d_\theta$-exact. Indeed, $\f=d_\theta\gamma$, where
$$
\gamma=\frac57 e^{12} -\frac37 e^{14} +\frac37 e^{23} -\frac17 e^{34} -e^{56}.
$$

As in the previous example, we have an almost nilpotent Lie group $Q=\R\ltimes_\mu N$, where $Q$ is the simply connected Lie group with solvable Lie algebra 
$\mathfrak{q}$ and $\mu$  is the unique smooth action of $\R$ on $N$ such that $\mu(t)_{*e} = {\rm exp}(tD)$, for any $t\in\R$. 
With respect to the rational basis $\{X_1,\ldots,X_6\}$ of $\mathfrak{n}$ given by
$X_1 = -\frac{1}{\sqrt{2}}e_2+e_4,  X_2 = -\frac{1}{\sqrt{2}}e_1+e_3,  X_3 = \frac{1}{\sqrt{2}}e_1+e_3,  X_4 = \frac{1}{\sqrt{2}}e_2+e_4, X_5 = \sqrt{2}e_5,  X_6 = \sqrt{2}e_6,$
the matrix associated to ${\rm exp}\left(\sqrt{2}D\right)$ is integer. More in detail, we have
$$
{\rm exp}\left(\sqrt{2}D\right) = {\rm diag}(-2,-2,2,2,0,0).
$$
Thus, ${\rm exp}^N(\Z\langle X_1,\ldots,X_6\rangle)$ is a lattice of $N$ preserved by  $\mu \left(\sqrt{2}\right)$ and, as a consequence,  
$$
\Gamma = \sqrt{2}~ \Z \ltimes_\mu {\rm exp}^N(\Z\langle X_1,\ldots,X_6\rangle)
$$
is a lattice in $Q$. The compact quotient $\Gamma \backslash Q$ is then a compact solvmanifold endowed with an invariant locally conformal calibrated $G_2$-structure $\f$.
\end{ex}

\section{Characterization as fiber bundles over $S^1$ }
Let us now consider a seven-dimensional compact manifold $M$ endowed with a locally conformal calibrated $G_2$-structure $\f$. 
We will show that if $\mathcal{L}_{X}  \varphi =0$, where $X={\theta^\#}$ is the $g_{\varphi}$-dual of  the Lee form $\theta$, then $M$ is fibered over $S^1$ and each fiber is endowed 
with a coupled $\SU(3)$-structure given by  the restriction of $(\omega = i_{X} \varphi, \psip = d\omega)$ to the fiber. 

We begin recalling the following result, which we will use later.
\begin{prop}[\cite{CLSS}]\label{CLSS}
Let $V$ be a seven-dimensional real vector space and $\f\in \Lambda^3(V^*)$ a stable 3-form which induces the inner product  $g_\f$ on $V$.  
Moreover, let ${\rm {\bf n}}\in V$ be a unit vector with $g_\f({\rm {\bf n}},{\rm {\bf n}})=1$ and let $W:=\langle {\rm {\bf n}}\rangle^\bot$ denote the $g_\f$-orthogonal complement of 
$\R{\rm {\bf n}}$. 
Then the pair $(\omega,\psi_+)$ defined by
$$\omega = (i_{\rm {\bf n}}  \f)_{|W},\quad \psip= \f_{|W}$$
is a pair of compatible, normalized, stable forms. The inner product  $h$ induced by this pair on $W$ satisfies $h = {g_\f}_{|W}$ and the stabilizer is $\SU(3)$.
\end{prop}

The next two lemmas will be useful to prove one part of the main theorem of this section.
\begin{lemma}\label{diseq}
Let $(M,g)$ be a Riemannian manifold and consider two differential forms $\theta\in\Lambda^1(M), \omega\in\Lambda^2(M)$. Then
$$
|\theta\W\omega|^2 = 3|\theta|^2|\omega|^2 - 6|u|^2,
$$
where $|\cdot|$ is the pointwise norm induced by $g$ and $u\in\Lambda^1(M)$ is defined locally as $u = u_idx^i$, $u_i = g^{rk}\theta_r\omega_{ki}$. From this follows
$$
|\theta\W\omega|^2 \leq 3|\theta|^2|\omega|^2.
$$
Moreover, with respect to the $L^2$ norm $\left\| \cdot \right\|$ induced by the $L^2$ inner product on $k$-forms 
$\langle \alpha,\beta\rangle = \int_M\alpha\W*\beta = \int_Mg(\alpha,\beta)*1$, we have
$$
\left\| \theta\W\omega \right\|^2\leq3\int_M|\theta|^2|\omega|^2*1.
$$
\end{lemma}
\proof
Let us recall some definitions in order to clarify the conventions we use. Given a $k$-covariant tensor $\eta$, we define the antisymmetrization of $\eta$ as
$$
{\rm Alt}(\eta) = \frac{1}{k!}\sum_{\sigma\in\mathfrak{S}_k}|\sigma|\eta^\sigma,
$$
where $|\sigma|$ is the sign of the permutation $\sigma$ and given any $k$ vectors $X_{i_1},\ldots,X_{i_k}\in\mathfrak{X}(M)$ we have 
$\eta^\sigma(X_{i_1},\ldots,X_{i_k}) = \eta(X_{i_{\sigma(1)}},\ldots,X_{i_{\sigma(k)}})$.
Given the differential forms $\alpha\in\Lambda^r(M), \beta\in\Lambda^s(M)$, we define the wedge product $\W : \Lambda^r(M)\times\Lambda^s(M)\rightarrow\Lambda^{r+s}(M)$ as
$$
\alpha\W\beta = \frac{(r+s)!}{r!s!}{\rm Alt}(\alpha\otimes\beta).
$$
In local coordinates we then have 
$$
(\theta\W\omega)_{ijk} = \theta_i\omega_{jk} - \theta_j\omega_{ik}+\theta_k\omega_{ij}.
$$
We can now start with our computations:
$$
\begin{array}{lcl}
|\theta\W\omega|^2 &=& (\theta\W\omega)_{ijk}g^{ia}g^{jb}g^{kc}(\theta\W\omega)_{abc}\\
                                &=& 3\theta_i\omega_{jk}g^{ia}g^{jb}g^{kc}\theta_a\omega_{bc} -6\theta_i\omega_{jk}g^{ia}g^{jb}g^{kc}\theta_b\omega_{ac}\\
                                &=& 3(\theta_ig^{ia}\theta_a)(\omega_{jk}g^{jb}g^{kc}\omega_{bc}) -6(g^{ia}\theta_i\omega_{ac})g^{ck}(g^{bj}\theta_b\omega_{jk})\\
                                &=& 3|\theta|^2|\omega|^2 - 6u_cg^{ck}u_k\\
                                &=& 3|\theta|^2|\omega|^2 - 6|u|^2.
\end{array}
$$
\endproof

For manifolds endowed with a $G_2$-structure we can prove the following 
\begin{lemma}\label{modom}
Let $M$ be a manifold endowed with a $G_2$-structure $\f$. Consider a vector field $X\in\mathfrak{X}(M)$ and define the 2-form $\omega:=i_X \f$. Then 
$$|\omega|^2 = 3|X|^2.$$
\end{lemma}
\proof
Using the identity $\f\W(i_X\f) = 2*(i_X\f)$, which reads $\f\W\omega = 2*\omega$ in our case, we have
$$
\begin{array}{lcl}
|\omega|^2*1 &=& \omega\W*\omega\\
                      &=& \frac12\omega\W\f\W\omega\\
                      &=& \frac12 (i_X\f)\W(i_X\f)\W\f\\
                      &=& 3|X|^2*1.
\end{array}
$$
\endproof

We can now prove the main result of this section.
\begin{teo}\label{teo?}
Let $M$ be a connected, compact, seven-dimensional manifold endowed with a locally conformal calibrated  $G_2$-structure $\f$ such that  $\mathcal{L}_X\f = 0$, 
where $X$ is the  $g_\f$-dual vector field of  the Lee form $\theta$.
Then 
\begin{enumerate}
\item $M$ is fibered over $S^1$ and each fiber is endowed with a coupled $\SU(3)$-structure given by the restriction of $(\omega= i_X \varphi, \psip = d\omega)$ to the fiber. 
\item  $M$ has  a locally conformal calibrated $G_2$-structure  $\hat \varphi$ such that  $d \hat \varphi =  - \hat \theta \wedge \hat\varphi$,  
where $\hat \theta$ is  a 1-form with integral periods. 
\end{enumerate}
\end{teo}

\begin{proof} 

(1) First of all, observe that since the closed 1-form $\theta$ is nowhere vanishing, we can consider the foliation  ${\mathcal F}_{\theta}$ generated by 
the integrable distribution $ker(\theta)$. 
We prove now that the pair $(\omega=i_X\f, \psip = d\omega)$ defines a coupled $\SU(3)$-structure when restricted to each leaf of this foliation, to do this we consider the 
tangent space to each point of the leaves of the foliation and apply Proposition \ref{CLSS}.

Under our hypothesis, we have a positive 3-form $\f$ such that $d\f=-\theta\W\f$, $X=\theta^\#$ and $\theta(X)\f = d\omega+\theta\W\omega$, 
where $\omega = i_X \f$ and $\theta(X)$ is a nonzero constant  (see Proposition  \ref{thetadiesis}). 
Let $L$ be a leaf of the foliation ${\mathcal F}_{\theta}$, then for any $p\in L$ we have 
$$
T_pL = \{Y_p\in T_pM\ |\ \theta_p(Y_p)=0 \} \subset T_pM
$$
and since $\theta(\cdot) = g_\f(X,\cdot)$, we also have
$$
ker(\theta) = \langle X\rangle^\bot.
$$ 
Thus, $T_pL = \langle X_p\rangle^\bot$ is a six-dimensional subspace of $T_pM$. Moreover, we can suppose that $X$ is normalized since it is nowhere zero, 
then $\theta(X) = |X|^2 =1$ and by Proposition \ref{CLSS} we have that the pair $((i_{X_p} \f)_{|T_pL},\ \f_{|T_pL})$ defines an $\SU(3)$-structure on $T_pL$. Now,
$$
(i_{X_p} \f)_{|T_pL} = {\omega_p}_{|T_pL} 
$$
and for any choice of tangent vectors $Q_p,Y_p,Z_p\in T_pL$ we have 
$$
\begin{array}{lcl}
\f_p(Q_p,Y_p,Z_p) &=& (d\omega_p+\theta_p\W\omega_p)(Q_p,Y_p,Z_p)\\
			      &=& (d\omega_p)(Q_p,Y_p,Z_p),
\end{array}
$$
since $\theta_p$ evaluated on any vector of $T_pL$ is zero. 
We then obtain  that
$$
\f_{|T_pL} = d\omega_{|T_pL}.
$$
Summarizing, the pair $(\omega, d\omega)$ defines a coupled $\SU(3)$-structure when restricted to each leaf $L$ of the foliation.

Let us now observe that $M$ is fibered over $S^1$, since it is compact and there is a  globally defined closed 1-form $\theta$ which is nowhere vanishing.
Using the same argument  as in \cite{Tischler}, we can approximate the $1$-form $\theta$ by $q \hat \theta$ for some integer $q$ and some $1$-form $\hat \theta$ with integral periods. 
Indeed, using the fact there is a bijection between $H^1(M,\Z)$ and the set of homotopy classes of maps from $M$ into the circle,
Tischler showed that there exists a smooth map $\pi: M \rightarrow S^1$ 
such that $\hat \theta = \pi^* (dt)$, where $dt$ is the length form on $S^1$.   
The form $q \hat  \theta$ (and therefore $\hat \theta$) has no zeroes since it is close to $\theta$. Thus $\pi^* (dt)$ has no zeros and then $\pi$ is a submersion. 
Since $M$ is compact, $\pi: M \rightarrow S^1$ is a fibration. The fibers are defined by the equation $\hat \theta =0 = q \hat \theta$, which is close to the equation $\theta =0$. 
Thus the tangent spaces to the fibers are close to the tangent spaces to the leaves.

We can now show that the restriction of $\omega$ and $d \omega$ to the fibers of $\pi$ defines a coupled $\SU(3)$-structure. 
Let $p$ be a point in a fiber $F$ of $\pi$ and let $L$ be the leaf of the foliation ${\mathcal F}_{\theta}$ such that $p\in L$. 
Using normal coordinates in $M$ with respect to the Riemannian metric $g_{\varphi}$, we get a diffeomorphism $\Phi$ from a neighborhood $D_F$ of $p$ in  the fiber $F$  
to a  neighborhood $D_L$ of $p$ in the leaf  $L$ such that $\Phi (p) = p$ and ${\Phi_*}_{p}$ is arbitrarily close to the identity. 
Since $(\omega, d \omega)$ defines  a coupled $\SU(3)$-structure when restricted to the leaf $L$, there exists an orthonormal basis  $\left\{e^1, \ldots, e^6\right\}$ 
of $T_p^*L$ such that
$$
\omega = e^{12} + e^{34} + e^{56}, \quad \psi_+ = e^{135} - e^{146} - e^{236} - e^{245}.
$$
Considering the basis  $\left\{\ff^1 = \Phi^* e^1, \ldots , \ff^6 = \Phi^* e^6\right\}$ of $T_p^*F$, 
we have that  $(\Phi^* (\omega), \Phi^*(d \omega))$ defines a coupled $\SU(3)$-structure on $T_p F$.\\
\ \\
(2) Note that from Lemma \ref{modom} and using the fact that $X$ is normalized we obtain 
$$
\begin{array}{c}
|\omega|^2 = 3|X|^2 = 3. 
\end{array}$$
Define the 3-form $\hat\f := d\omega +  q \hat \theta \W\omega$, it is a positive 3-form. Indeed, using Lemma \ref{diseq} and the previous observation we have
$$\renewcommand\arraystretch{1.6}
\begin{array}{lcl}
| \hat\f-\f |^2 &=& |(q\hat\theta-\theta)\W\omega |^2\\
                &\leq& 3|q\hat\theta-\theta|^2|\omega|^2\\
                 &=& 9|q\hat\theta-\theta|^2.
\end{array}
$$
Then
$$\left\| \hat\f-\f\right\|^2 = \int_M| \hat\f-\f |^2*1\leq 9\int_M|q\hat\theta-\theta|^2*1 = 9\left\|q\hat\theta-\theta\right\|^2.$$
Observe that $\| \hat\f-\f\|$ is arbitrary small since $\left\|q\hat\theta-\theta\right\|$ is, therefore the 3-form $\hat\f$ is positive since it lies in an arbitrary small neighborhood of the 
positive 3-form $\f$ and being a positive 3-form is an open condition.  Therefore, since $d \hat \f =  - q \hat \theta \wedge \hat \f$, the $3$-form $\hat \f$ defines a  new locally conformal 
calibrated $G_2$-structure with associated Lee form $q \hat \theta$, which is a 1-form with integral periods.
\end{proof}

\begin{remark} By the previous theorem we have  that $M = N_{\nu}$ is the mapping torus of a six-dimensional manifold, but the diffeomorphism $\nu$  in general does not preserve the 
coupled $\SU(3)$-structure on the fiber.
\end{remark} 

Note that the  compact locally conformal calibrated $G_2$-manifold  $M = \Gamma \backslash G$ obtained in Example \ref{excpt} satisfies the assumptions 
of previous theorem. Indeed, it admits a locally conformal calibrated  $G_2$-structure $\f$ such that  $\mathcal{L}_X\f = 0$, 
where $X$ is the  $g_\f$-dual vector field of  the Lee form $\theta$.
Therefore, $M$ is fibered over $S^1$ and each fiber is endowed with a coupled $\SU(3)$-structure given by the restriction of $(-\omega,d(-\omega))$ to the fiber.

\bigskip

\noindent {\bf Acknowledgments.} This work has been partially supported by (Spanish) MINECO Project MTM2011-28326-C02-02,
Project UPV/EHU ref.\ UFI11/52 and by (Italian) Project PRIN ``Variet\`a reali e complesse: geometria, topologia e analisi armonica'', 
Project FIRB ``Geometria Differenziale e Teoria Geometrica delle Funzioni'', and by GNSAGA of INdAM.

\end{document}